\documentclass[a4paper]{amsart}
\usepackage{amsthm,amsfonts,amsmath,amssymb}
\usepackage[abs]{overpic}
\usepackage{comment} 

\newtheorem{theorem}{Theorem}[section]
\newtheorem{lemma}[theorem]{Lemma}

\newtheorem{proposition}[theorem]{Proposition}
\newtheorem{corollary}[theorem]{Corollary}

\theoremstyle{definition}

\theoremstyle{remark}

\newcommand{\e}{\varepsilon}

\newcommand{\Z}{\mathbb{Z}}

\makeatletter

\@addtoreset{figure}{section}
\makeatother

\makeatletter
  
  \@addtoreset{equation}{section}
\makeatother

\begin{document}
\title[Writhes and $2k$-moves for virtual knots]{Writhes and $2k$-moves for virtual knots}

\author[Kodai Wada]{Kodai Wada}
\address{Department of Mathematics, Kobe University, 1-1 Rokkodai-cho, Nada-ku, Kobe 657-8501, Japan}
\email{wada@math.kobe-u.ac.jp}

\makeatletter
\@namedef{subjclassname@2020}{%
  \textup{2020} Mathematics Subject Classification}
\makeatother


\thanks{This work was supported by JSPS KAKENHI Grant Numbers JP21K20327 and  JP23K12973.}



\begin{abstract}
A $2k$-move is a local deformation adding or removing $2k$ half-twists. 
We show that if two virtual knots are related by a finite sequence of $2k$-moves, then their $n$-writhes are congruent modulo~$k$ for any nonzero integer~$n$, and their odd writhes are congruent modulo~$2k$. 
Moreover, we give a necessary and sufficient condition for two virtual knots to  have the same congruence class of odd writhes modulo~$2k$. 
\end{abstract}

\maketitle

\section{Introduction}\label{sec-intro}
Let $k$ be a positive integer. 
A \emph{$2k$-move} on a knot diagram is a local deformation adding or removing $2k$ half-twists as shown in Figure~\ref{2k-move}. 
A $2$-move is equivalent to a crossing change; that is, 
a $2$-move is realized by a crossing change, and vice verse. 
In this sense, a $2k$-move can be considered as a generalization of a crossing change. 
The $2k$-moves form an important family of local moves in classical knot theory. 
In fact, they have been well studied by means of many invariants of classical knots and links; 
for example, Alexander polynomials~\cite{Kin}, Jones, HOMFLYPT and Kauffman polynomials~\cite{Prz}, Burnside groups~\cite{DP02,DP04} and  Milnor invariants~\cite{MWY}. 

\begin{figure}[htbp]
\centering
  \begin{overpic}[width=7cm]{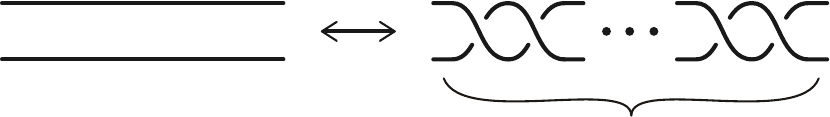}
    \put(81,25){$2k$}
    \put(119.5,-12){$2k$ half-twists}
  \end{overpic}
\vspace{1em}
\caption{A $2k$-move}
\label{2k-move}
\end{figure}

Recently, Jeong, Choi and Kim~\cite{JCK} extended the classical study of $2k$-moves to the setting of virtual knots, 
which are a generalization of classical knots discovered by Kauffman~\cite{Kau99}. 
Roughly speaking, a \emph{virtual knot} is an equivalence class of generalized knot diagrams called \emph{virtual knot diagrams} under seven types of local deformations. 
We say that two virtual knots $K$ and $K'$ are related by a $2k$-move if a diagram of $K'$ is a result of a $2k$-move on a diagram of $K$. 

For a virtual knot $K$, Kauffman~\cite{Kau04} defined an integer-valued invariant $J(K)$ called the \emph{odd writhe}. 
Satoh and Taniguchi~\cite{ST} generalized it to a sequence of integer-valued invariants $J_{n}(K)$ of $K$ called the \emph{$n$-writhes}. 
This sequence $\{J_{n}(K)\}_{n\ne0}$ gives rise to a polynomial invariant $P_{K}(t)$ of $K$ known as the the \emph{affine index polynomial}~\cite{Kau13}, 
which is essentially equivalent to the \emph{writhe polynomial}~\cite{CG}. 
Refer to~\cite{CFGMX} for a good survey of virtual knot invariants derived from chord index, including $J(K)$, $J_{n}(K)$ and $P_{K}(t)$. 

In~\cite[Theorem~2.3]{JCK}, Jeong, Choi and Kim established a necessary condition for two virtual knots to be related by a finite sequence of $2k$-moves using their affine index polynomials. 
Examining the proof of~\cite[Theorem~2.3]{JCK}, we can find another necessary condition for such a pair of virtual knots in terms of the $n$-writhes and odd writhes. 
The first aim of this paper is to prove the following. 

\begin{theorem}\label{thm-writhe}
If two virtual knots $K$ and $K'$ are related by a finite sequence of $2k$-moves, 
then the following hold: 
\begin{enumerate}
\item
$J_{n}(K)\equiv J_{n}(K')\pmod{k}$ for any nonzero integer $n$. 

\item 
$J(K)\equiv J(K')\pmod{2k}$. 
\end{enumerate}
\end{theorem}

Although the assertion~(i) in this theorem is obtained from~\cite[Theorem~2.3]{JCK} immediately, 
the author believes that it is worth stating it as a separate result. 

A \emph{$\Xi$-move} on a virtual knot diagram is a local deformation 
exchanging the positions of $c_{1}$ and $c_{3}$ of three consecutive real crossings $c_{1}$, $c_{2}$ and $c_{3}$ as shown in Figure~\ref{Xi-move}, 
where we omit the over/under information of each crossing $c_{i}$ $(i=1,2,3)$. 
The $\Xi$-move arises naturally as a diagrammatic characterization of virtual knots having the same odd writhe. 
In fact, Satoh and Taniguchi~\cite{ST} proved the following. 

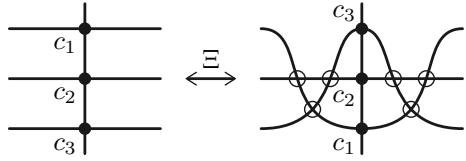
\begin{figure}[htbp]
\centering
  \begin{overpic}[width=6cm]{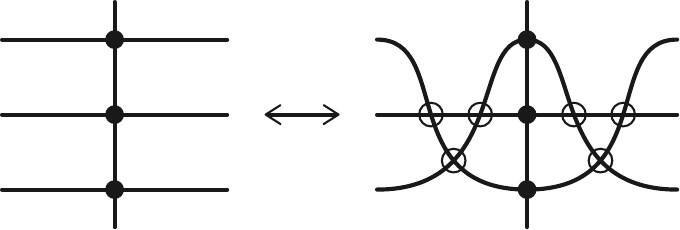}
    \put(72.75,33){$\Xi$}
    \put(17,40){$c_{1}$}
    \put(17,21){$c_{2}$}
    \put(17,2){$c_{3}$}
    \put(121.5,52){$c_{3}$}
    \put(121.5,20){$c_{2}$}
    \put(121.5,2){$c_{1}$}
  \end{overpic}
\caption{A $\Xi$-move}
\label{Xi-move}
\end{figure}

\begin{theorem}[{\cite[Theorem 1.7]{ST}}]\label{thm-ST}
For two virtual knots $K$ and $K'$, the following are equivalent: 
\begin{enumerate}
\item
$J(K)=J(K')$. 
\item
$K$ and $K'$ are related by a finite sequence of $\Xi$-moves.
\end{enumerate}
\end{theorem}

Inspired by this theorem, we use $\Xi$-moves together with $2k$-moves to characterize virtual knots having the same congruence class of odd writhes modulo~$2k$. 
The second aim of this paper is to prove the following. 

\begin{theorem}\label{thm-2kXi}
For two virtual knots $K$ and $K'$, the following are equivalent: 
\begin{enumerate}
\item
$J(K)\equiv J(K')\pmod{2k}$. 
\item
$K$ and $K'$ are related by a finite sequence of $2k$-moves and $\Xi$-moves.
\end{enumerate}
\end{theorem}

Although the crossing change is an unknotting operation for classical knots, 
it was shown by Carter, Kamada and Saito in~\cite[Proposition~25]{CKS} that not every virtual knot can be unknotted by crossing changes. 
This fact justifies the notion of flat virtual knots. 
A \emph{flat virtual knot}~\cite{Kau99} is an equivalence class of virtual knots up to crossing changes. 
Equivalently, a flat virtual knot is represented by a virtual knot diagram with all the real crossings replaced by flat crossings,  
where a \emph{flat crossing} is a transverse double point with no over/under information. 

In~\cite[Lemma~2.2]{Che}, Cheng showed that 
the odd writhe for any virtual knot takes values in even integers. 
Therefore any virtual knot $K$ and the trivial one $O$ satisfy $J(K)\equiv J(O)\equiv0\pmod{2}$. 
By Theorem~\ref{thm-2kXi} for $k=1$, the two knots $K$ and $O$ are related by a finite sequence of crossing changes and $\Xi$-moves. 
In other words, we have the following. 

\begin{corollary}
Any flat virtual knot can be deformed into the trivial knot by a finite sequence of flat $\Xi$-moves; 
that is, the flat $\Xi$-move is an unknotting operation for flat virtual knots. 
Here, a flat $\Xi$-move is a $\Xi$-move with all the real crossings replaced by flat ones. 
\qed
\end{corollary}

The rest of this paper is organized as follows. 
In Section~\ref{sec-thm1}, we review the definitions of a virtual knot, a Gauss diagram, the $n$-writhe and the odd writhe, and prove Theorem~\ref{thm-writhe}. 
Section~\ref{sec-thm2} is devoted to the proof of Theorem~\ref{thm-2kXi}. 
Our main tool is the notion of shell-pairs, which are certain pairs of chords of a Gauss diagram introduced in~\cite{MSW}. 
In the last section, for two virtual knots $K$ and $K'$ that are related by a finite number of $2k$-moves, we study their $2k$-move distance $\mathrm{d}_{2k}(K,K')$ defined as the minimal number of such $2k$-moves. 
We show that for any virtual knot $K$ and any positive integer $a$, there is a virtual knot $K'$ such that $\mathrm{d}_{2k}(K,K')=a$ (Proposition~\ref{prop-distance}).

\section{Proof of Theorem~\ref{thm-writhe}}\label{sec-thm1}
We begin this section by recalling the definitions of virtual knots and Gauss diagrams from~\cite{GPV,Kau99}. 

A \emph{virtual knot diagram} is the image of an immersion of an oriented circle into the plane 
whose singularities are only transverse double points. 
Such double points consist of \emph{positive}, \emph{negative} and \emph{virtual crossings} as shown in Figure~\ref{xing}. 
A positive/negative crossing is also called a \emph{real crossing}. 

\begin{figure}[htbp]
\centering
  \begin{overpic}[width=6cm]{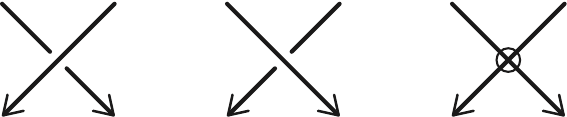}
    \put(0.8,-12){positive}
    \put(67.25,-12){negative}
    \put(139,-12){virtual}
  \end{overpic}
\vspace{1em}
\caption{Types of double points}
\label{xing}
\end{figure}

Two virtual knot diagrams are said to be \emph{equivalent} if they are related by a finite sequence of \emph{generalized Reidemeister moves} I--VII as shown in Figure~\ref{gReid-move}. 
A \emph{virtual knot} is the equivalence class of a virtual knot diagram. 
In particular, a classical knot can be considered as a virtual knot diagram with no virtual crossings, called a \emph{classical knot diagram}, up to the moves I, II and III. 
In~\cite[Theorem~1.B]{GPV}, Goussarov, Polyak and Viro proved that two equivalent classical knot diagrams are related by a finite sequence of moves I, II, and III; 
in other words, the set of virtual knots contains that of classical knots.  
In this sense, virtual knots are a generalization of classical knots. 

\begin{figure}[htbp]
\centering
  \begin{overpic}[width=12cm]{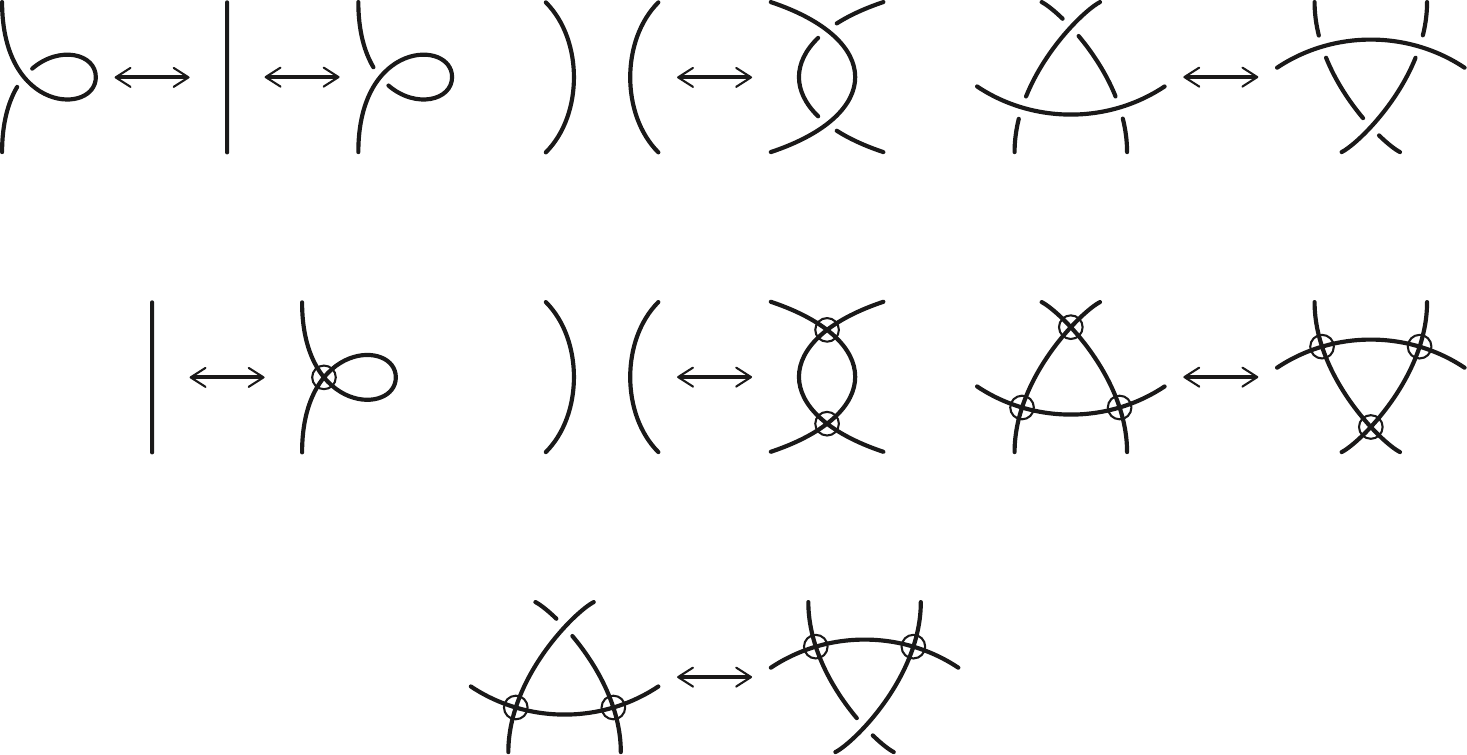}
    \put(33.5,162){I}
    \put(68.5,162){I} 
    \put(162.5,162){II}
    \put(278.5,162){III}
    \put(48,92){IV} 
    \put(162.5,92){V}
    \put(278.5,92){VI}
    \put(158.5,23){VII}
  \end{overpic}
\caption{Generalized Reidemeister moves I--VII}
\label{gReid-move}
\end{figure}
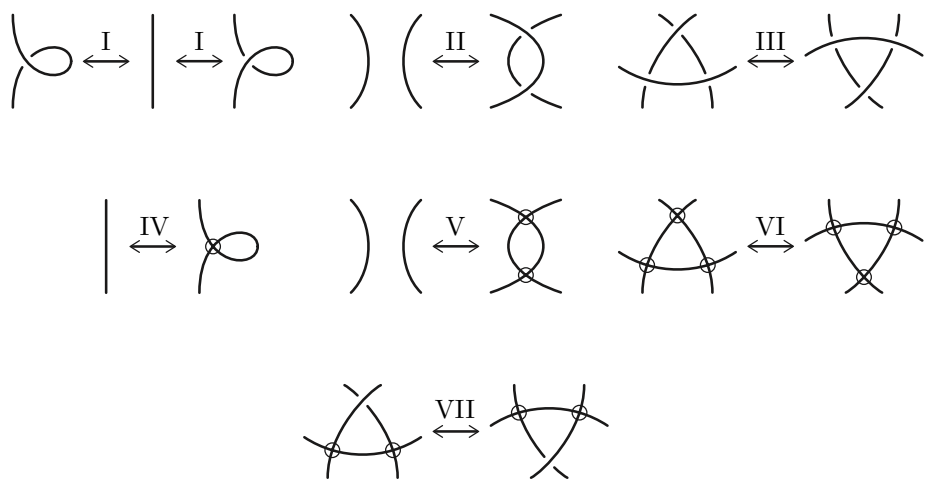

A \emph{Gauss diagram} is an oriented circle equipped with a finite number of signed and oriented chords whose endpoints lie disjointly on the circle. 
In figures the underlying circle and chords of a Gauss diagram will be drawn with thick and thin lines, respectively. 
Gauss diagrams provide an alternative way of representing virtual knots. 
For a virtual knot diagram $D$ with $n$ real crossings (and some or no virtual crossings), 
the \emph{Gauss diagram $G_{D}$ associated with $D$} is constructed as follows. 
It consists of a circle and $n$ chords connecting the preimage of each real crossing of $D$. 
Each chord of $G_{D}$ has the sign of the corresponding real crossing of $D$, 
and it is oriented from the overcrossing to the undercrossing. 
For a virtual knot $K$, a \emph{Gauss diagram of $K$} is defined to be a Gauss diagram associated with a virtual knot diagram of $K$. 

A motivation of introducing virtual knot theory comes from the realization of Gauss diagrams. 
In fact, the construction above defines a surjective map from virtual knot diagrams onto Gauss diagrams, 
although not every Gauss diagram can be realized by a classical knot diagram.  
Moreover, 
this map induces a bijection between the set of virtual knots and that of Gauss diagrams modulo Reidemeister moves~I, II and III defined in the Gauss diagram level as shown in Figure~\ref{Reid-Gauss}~\cite[Theorem~1.A]{GPV}. 
Refer also to~\cite[Section 3.2]{Kau99}. 
Therefore a virtual knot can be regarded as the equivalence class of a Gauss diagram. 

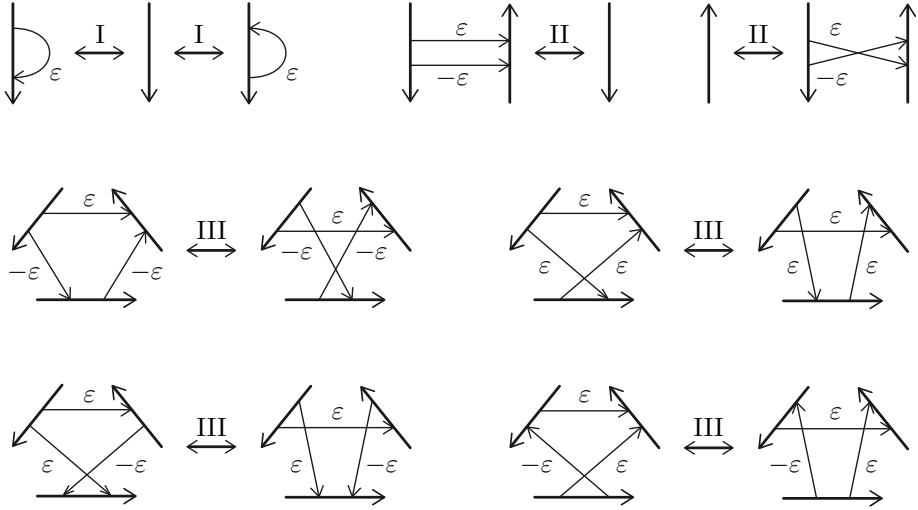
\begin{figure}[htbp]
\centering
  \begin{overpic}[width=12cm]{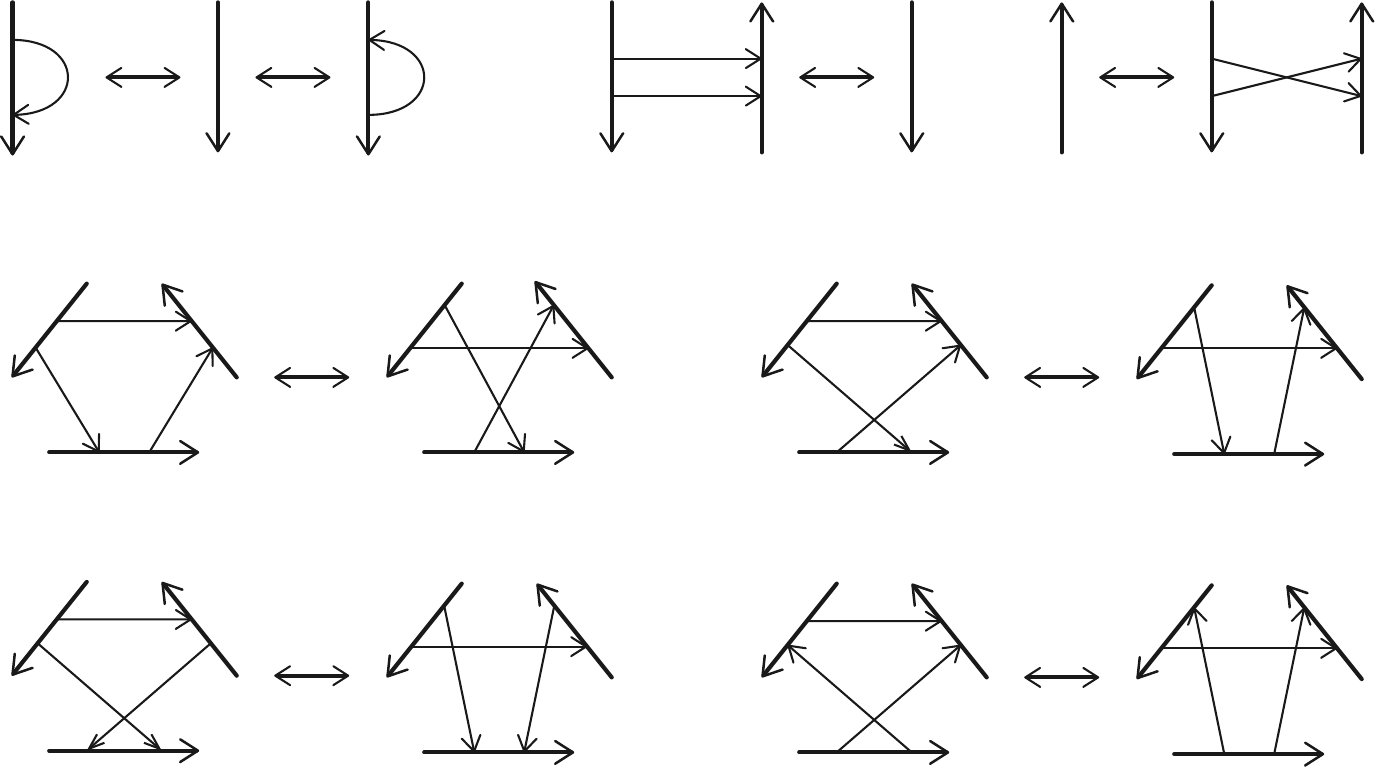}
    \put(34,175){I}
    \put(17,161){$\e$}
    \put(71,175){I}
    \put(105.5,161){$\e$}
    \put(204,175){II}
    \put(169,178){$\e$}
    \put(161.5,159){$-\e$}
    \put(278.5,175){II}
    \put(309,178){$\e$}
    \put(303.5,159){$-\e$}
    \put(71.8,101){III} 
    \put(29.5,113.5){$\e$}
    \put(1,86){$-\e$}
    \put(47,86){$-\e$}
    \put(122.5,107){$\e$}
    \put(103,93.5){$-\e$}
    \put(131,93.5){$-\e$}
    \put(257.9,101){III}
    \put(215,113.5){$\e$}
    \put(200,88){$\e$}
    \put(229,88){$\e$}
    \put(309,107){$\e$}
    \put(293,88){$\e$}
    \put(322.5,88){$\e$}
    \put(71.8,26.5){III}
    \put(29.5,39.5){$\e$}
    \put(14,13){$\e$}
    \put(41,13){$-\e$}
    \put(122.5,33){$\e$}
    \put(107,13){$\e$}
    \put(135,13){$-\e$}
    \put(257.9,26.5){III}
    \put(215,39.5){$\e$}
    \put(193,13){$-\e$}
    \put(229,13){$\e$}
    \put(309,33){$\e$}
    \put(286,13){$-\e$}
    \put(322.5,13){$\e$}
  \end{overpic}
\caption{Reidemeister moves I, II and III on Gauss diagrams ($\e=\pm1$)}
\label{Reid-Gauss}
\end{figure}

We will use two local deformations on Gauss diagrams as shown in Figure~\ref{2kXi-Gauss} as well as the Reidemeister moves I, II and III. 
These deformations are the counterparts of a $2k$-move and a $\Xi$-move for Gauss diagrams. 
More precisely, a \emph{$2k$-move on a Gauss diagram} adds or removes $2k$ chords with the same sign $\e$ whose initial and terminal endpoints appear alternatively. 
Let $P_{1}$, $P_{2}$ and $P_{3}$ be three consecutive endpoints of chords of a Gauss diagram. 
A \emph{$\Xi$-move} exchanges the positions of $P_{1}$ and $P_{3}$ with preserving the signs $\e_{1},\e_{2},\e_{3}$ and orientations of the chords. 
In the right of the figure, a pair of dots $\bullet$ marks the two endpoints $P_{1}$ and $P_{3}$ exchanged by a $\Xi$-move.

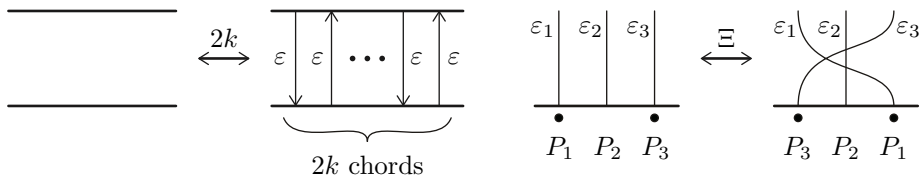
\begin{figure}[htbp]
\centering
  \begin{overpic}[width=12cm]{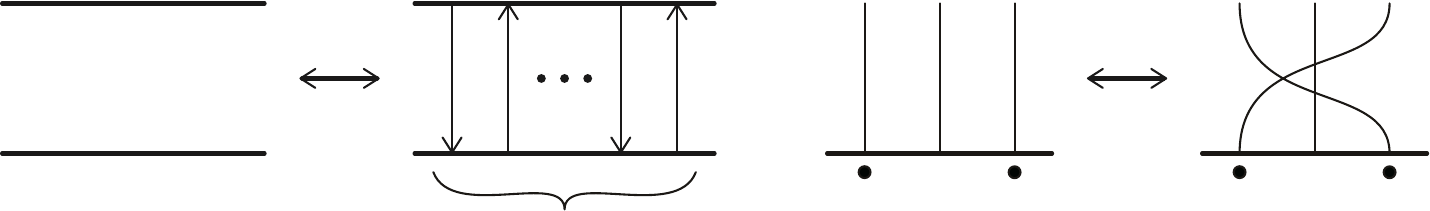}
    \put(76,36){$2k$}
    \put(114,-12){$2k$ chords}
    \put(100,29.5){$\e$}
    \put(114,29.5){$\e$}
    \put(151,29.5){$\e$}
    \put(165,29.5){$\e$}
    \put(266,36){$\Xi$}
    \put(196,41){$\e_{1}$}
    \put(214,41){$\e_{2}$}
    \put(232,41){$\e_{3}$}
    \put(287,41){$\e_{1}$}
    \put(303.5,41){$\e_{2}$}
    \put(333,41){$\e_{3}$}
    \put(201.5,-5){$P_{1}$}
    \put(219.5,-5){$P_{2}$}
    \put(237.5,-5){$P_{3}$}
    \put(290.5,-5){$P_{3}$}
    \put(308.5,-5){$P_{2}$}
    \put(326.5,-5){$P_{1}$}

  \end{overpic}
\vspace{1em}
\caption{A $2k$-move and a $\Xi$-move on Gauss diagrams}
\label{2kXi-Gauss}
\end{figure}

Using Gauss diagrams, we now define the $n$-writhe and the odd writhe of a virtual knot $K$. 
For a Gauss diagram $G$ of $K$, let $\gamma$ be a chord of $G$. 
If $\gamma$ has a sign~$\e$, then we assign $\e$ and $-\e$ to the terminal and initial endpoints of $\gamma$, respectively. 
The endpoints of $\gamma$ divides the underlying circle of $G$ into two arcs. 
Let $\alpha$ be one of the two oriented arcs that starts at the initial endpoint of $\gamma$. 
See Figure~\ref{arc}. 
The \emph{index} of $\gamma$, $\mathrm{ind}(\gamma)$, is the sum of sings of endpoints of chords on $\alpha$. 
For an integer~$n$, we denote by $J_{n}(G)$ the sum of signs of chords with index $n$. 
In~\cite[Lemma~2.3]{ST}, Satoh and Taniguchi proved that $J_{n}(G)$ is an invariant of the virtual knot $K$ for any $n\ne0$; 
that is, it is independent of the choice of $G$. 
This invariant is called the \emph{$n$-writhe} of $K$ and denoted by $J_{n}(K)$. 
The \emph{odd writhe $J(K)$} of $K$, defined by Kauffman~\cite{Kau04}, is given by 
\[
J(K)=\sum_{n\in\Z}J_{2n-1}(K). 
\] 
Refer to~\cite{CFGMX,Kau04,ST} for more details.  

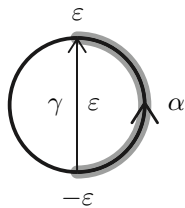
\begin{figure}[t]
\vspace{1em}
\centering
  \begin{overpic}[width=2cm]{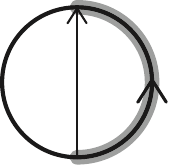}
    \put(15,26){$\gamma$}
    \put(30,26){$\e$}
    \put(60,26){$\alpha$}
    \put(24,60){$\e$}
    \put(20,-10){$-\e$}
  \end{overpic}
\vspace{1em}
\caption{A chord $\gamma$ with sign $\e$ and its specified arc $\alpha$}
\label{arc}
\end{figure}

In~\cite{JCK}, Jeong, Choi and Kim prepared Lemma~2.2 to show Theorem~2.3 giving the necessary condition mentioned in Section~\ref{sec-intro}. 
We rephrase and reprove this lemma from the Gauss diagram point of view for the proof of Theorem~\ref{thm-writhe}.

\begin{lemma}[{cf. \cite[Lemma~2.2]{JCK}}]\label{lem-writhe}
If two Gauss diagrams $G$ and $G'$ are related by a single $2k$-move, 
then there is a unique integer $n$ such that 
\[
J_{n}(G)-J_{n}(G')=\e k,\ 
J_{-n}(G)-J_{-n}(G')=\e k \text{ and } 
J_{m}(G)=J_{m}(G') 
\]
for some $\e=\pm1$ and any integer $m\neq \pm n$. 
\end{lemma}

\begin{proof}
We may assume that $G'$ is obtained from $G$ by removing $2k$ chords $\gamma_{1}$, $\gamma_{2},\dots,\gamma_{2k}$ with sing $\e$ involved in a $2k$-move as shown in Figure~\ref{2k-Gauss}, 
where we depict the signs of the initial and terminal endpoints of each $\gamma_{i}$ $(i=1,2,\ldots,2k)$, instead of omitting the sign of $\gamma_{i}$ itself. 

\begin{figure}[htbp]
\vspace{1em}
\centering
  \begin{overpic}[width=7cm]{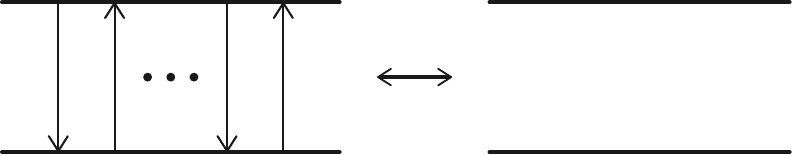}
    \put(99,24){$2k$}
    \put(39,-23){$G$}
    \put(157,-23){$G'$}
    \put(3,18){$\gamma_{1}$}
    \put(17.5,18){$\gamma_{2}$}
    \put(73,18){$\gamma_{2k}$}
    \put(7,42){$-\e$}
    \put(27,42){$\e$}
    \put(49,42){$-\e$}
    \put(69,42){$\e$}
    \put(12,-9){$\e$}
    \put(21,-9){$-\e$}
    \put(54.5,-9){$\e$}
    \put(63.5,-9){$-\e$}
  \end{overpic}
\vspace{2em}
\caption{$G'$ is a result of a $2k$-move on $G$ removing $2k$ chords}
\label{2k-Gauss}
\end{figure}
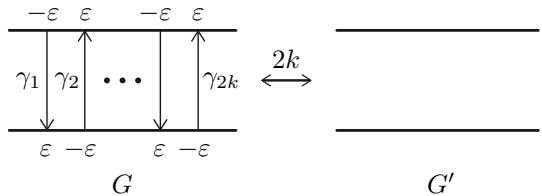

For any chord $\gamma$ of $G$ other than the $2k$ chords $\gamma_{1},\ldots,\gamma_{2k}$, 
let $\gamma'$ be the corresponding chord of $G'$. 
Then $\gamma$ and $\gamma'$ have the same index; that is,  
$\mathrm{ind}(\gamma)=\mathrm{ind}(\gamma')$. 
In fact, the sum of signs of $2k$ consecutive endpoints of $\gamma_{i}$'s on an arc equals zero. 

Let $n\in\Z$ (possibly zero) be the index of $\gamma_{1}$. 
Then we have $\mathrm{ind}(\gamma_{2i-1})=n$ and $\mathrm{ind}(\gamma_{2i})=-n$ for $i=1,2,\dots,k$. 
Moreover, the sum of signs of $k$ chords $\gamma_{1},\gamma_{3},\ldots,\gamma_{2k-1}$ is equal to $\e k$, 
and the sum of signs of $k$ chords $\gamma_{2},\gamma_{4},\ldots,\gamma_{2k}$ is also $\e k$. 
Therefore we obtain the conclusion. 
\end{proof}

\begin{proof}[Proof of Theorem~\ref{thm-writhe}]
Assume that $K$ and $K'$ are related by a single $2k$-move, 
and let $G$ and $G'$ be Gauss diagrams of $K$ and $K'$, respectively. 

Since $G$ and $G'$ are related by a combination of a single $2k$-move and Reidemeister moves,  
it follows from Lemma~\ref{lem-writhe} that 
\[
J_{n}(K)-J_{n}(K')=J_{n}(G)-J_{n}(G')\in\{0,\pm k\}
\]
for any nonzero integer~$n$. 
Therefore the assertion~(i) holds. 

If there is an odd integer $n$ such that $J_{n}(K)-J_{n}(K')=\e k$ holds for some $\e=\pm1$, 
then we have 
\[
J_{-n}(K)-J_{-n}(K')=\e k \text{ and } J_{m}(K)=J_{m}(K')
\]
for any nonzero integer $m\neq \pm n$. 
Otherwise, we have $J_{m}(K)=J_{m}(K')$ for any odd integer~$m$. 
In both cases, it follows that $J(K)-J(K')\in\{0,\pm2k\}$. 
This completes the proof for the assertion~(ii). 
\end{proof}

Theorem~\ref{thm-writhe} together with Theorem~\ref{thm-ST} implies an interesting consequence, which states that 
the set of $2k$-move equivalence classes of a virtual knot $K$ for all $k\geq1$ determines the $\Xi$-move equivalence class of $K$. 

\begin{proposition}
If two virtual knots $K$ and $K'$ are related by a finite sequence of $2k$-moves for all $k\geq1$, then $K$ and $K'$ are related by a finite sequence of $\Xi$-moves. 
\end{proposition}

\begin{proof}
By assuming that $K$ and $K'$ satisfy $J(K)\ne J(K')$, there is a positive integer $k$ such that $J(K)\not\equiv J(K')\pmod{2k}$. 
By Theorem~\ref{thm-writhe}(ii), 
this contradicts that $K$ and $K'$ are related by a finite sequence of $2k$-moves for all $k\geq1$. 
Since we have $J(K)=J(K')$, Theorem~\ref{thm-ST} gives the conclusion. 
\end{proof}

\section{Proof of Theorem~\ref{thm-2kXi}}\label{sec-thm2}
This section is devoted to the proof of Theorem~\ref{thm-2kXi}. 
Our main tool is the notion of shell-pairs, which are certain pairs of chords of a Gauss diagram developed in~\cite{MSW} for classifying $2$-component virtual links up to $\Xi$-moves. 

Let $P_{1}$, $P_{2}$ and $P_{3}$ be three consecutive endpoints of chords of a Gauss diagram~$G$. 
We say that a chord of $G$ is a \emph{shell} if it connects $P_{1}$ and $P_{3}$. 
See the left of Figure~\ref{shell}. 
Note that the orientation of a shell can be reversed by a $\Xi$-move exchanging the positions of $P_{1}$ and $P_{3}$. 
A \emph{positive shell-pair} (or \emph{negative shell-pair}) consists of a pair of positive shells (or \emph{negative shells}) whose four endpoints are consecutive. 
See the right of the figure, where we omit the orientations of shells. 

\begin{figure}[htbp]
\centering
  \begin{overpic}[width=9cm]{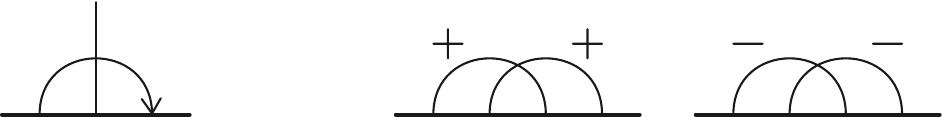}
    \put(11,16){$\e$}
    \put(6,-12){$P_{1}$}
    \put(21,-12){$P_{2}$}
    \put(36,-12){$P_{3}$}
    \put(124,-12){positive}
    \put(205,-12){negative}
   \end{overpic}
\vspace{1em}
\caption{A shell and a positive/negative shell-pair}
\label{shell}
\end{figure}
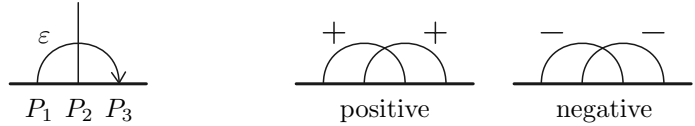

We prepare three results (Lemmas~\ref{lem-shellpair}, \ref{lem-censecutive} and Proposition~\ref{prop-normalform}) to give the proof of Theorem~\ref{thm-2kXi}. 
The first and second results will be used to prove the third one. 

The following lemma was shown in \cite{MSW,ST}. 

\begin{lemma}[{\cite[Section~4]{MSW}, \cite[Fig. 13]{ST}}]\label{lem-shellpair}
Let $G$, $G'$ and $G''$ be Gauss diagrams. 
\begin{enumerate}
\item 
If $G'$ is obtained from $G$ by a local deformation 
exchanging the positions of a shell-pair and an endpoint of a chord in $G$ 
with preserving the orientations of the chords 
as shown in the top of Figure~\ref{shellpair-move}, then $G$ and $G'$ are related by a finite sequence of $\Xi$-moves and Reidemeister moves. 
\item 
If $G''$ is obtained from $G$ by a local deformation 
adding or removing two consecutive shell-pairs with opposite signs as shown in the bottom of Figure~\ref{shellpair-move}, then $G$ and $G''$ are related by a finite sequence of $\Xi$-moves and Reidemeister moves. 
\end{enumerate}
\end{lemma}

\begin{figure}[htbp]
\centering
  \begin{overpic}[width=8cm]{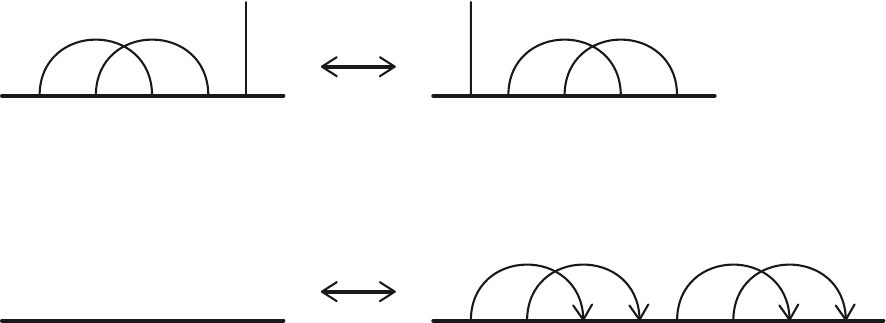}
    \put(13.5,74){$\e$}
    \put(46,74){$\e$}
    \put(134,74){$\e$}
    \put(166,74){$\e$}
    \put(125,17){$\e$}
    \put(157,17){$\e$}
    \put(172,17){$-\e$}
    \put(203,17){$-\e$}
   \end{overpic}
\caption{Local deformations in Lemma~\ref{lem-shellpair}}
\label{shellpair-move}
\end{figure}
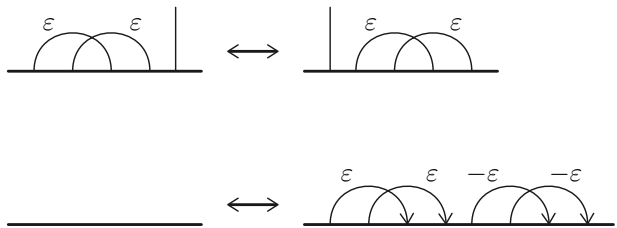

\begin{lemma}\label{lem-censecutive}
Let $G$ and $G'$ be Gauss diagrams and $k$ a positive integer. 
If $G'$ is obtained from $G$ by a local deformation adding or removing $k$ consecutive shell-pairs with the same sign~$\e$ as shown in Figure~\ref{censec-shellpairs}, then $G$ and $G'$ are related by a finite sequence of $2k$-moves, $\Xi$-moves and Reidemeister moves. 
\end{lemma}

\begin{figure}[htbp]
\centering
\vspace{1em}
  \begin{overpic}[width=9cm]{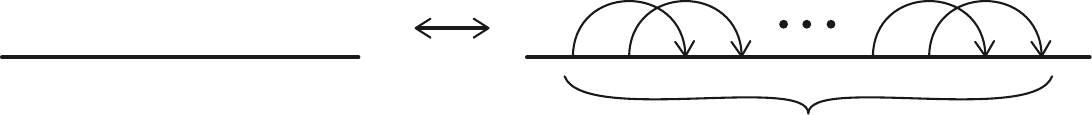}
    \put(139,29){$\e$}
    \put(165,29){$\e$}
    \put(209,29){$\e$}
    \put(235,29){$\e$}
    \put(164,-12){$k$ shell-pairs}
   \end{overpic}
\vspace{1em}
\caption{Adding or removing $k$ consecutive shell-pairs}
\label{censec-shellpairs}
\end{figure}
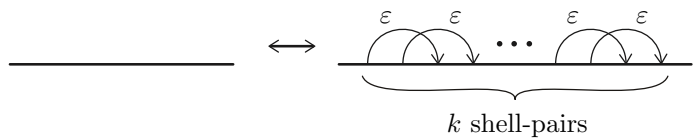

\begin{proof}
We only prove the result for $k=2$. 
The other cases are shown similarly. 

Assume that $G'$ is obtained from $G$ by adding two consecutive shell-pairs with sign~$\e$. 
The proof follows from Figure~\ref{pf-lem-censecutive}, 
which gives a sequence of Gauss diagrams 
\[
G=G_{0}, G_{1},\ldots,G_{6}=G'
\] 
such that for each $i=1,2,\ldots,6$, 
$G_{i}$ is obtained from $G_{i-1}$ by a combination of $2k$-moves, $\Xi$-moves and Reidemeister moves. 
More precisely, 
we obtain $G_{1}$ from $G_{0}=G$ by a Reidemeister move I adding a positive chord, 
$G_{2}$ from $G_{1}$ by a $4$-move adding four chords with sign $\e$, 
and $G_{3}$ from $G_{2}$ by a $\Xi$-move exchanging the positions of the two endpoints with dots $\bullet$. 
By Lemma~\ref{lem-shellpair}(i), we can move the resulting shell-pair (with preserving the orientations of the chords) to get $G_{4}$ from $G_{3}$. 
After deforming $G_{4}$ into $G_{5}$ by a $\Xi$-move, 
we finally obtain $G_{6}=G'$ by two $\Xi$-moves reversing the orientations of shells and a Reidemeister move I removing a positive chord. 
\end{proof}

\begin{figure}[t]
\centering
\vspace{1em}
  \begin{overpic}[width=11cm]{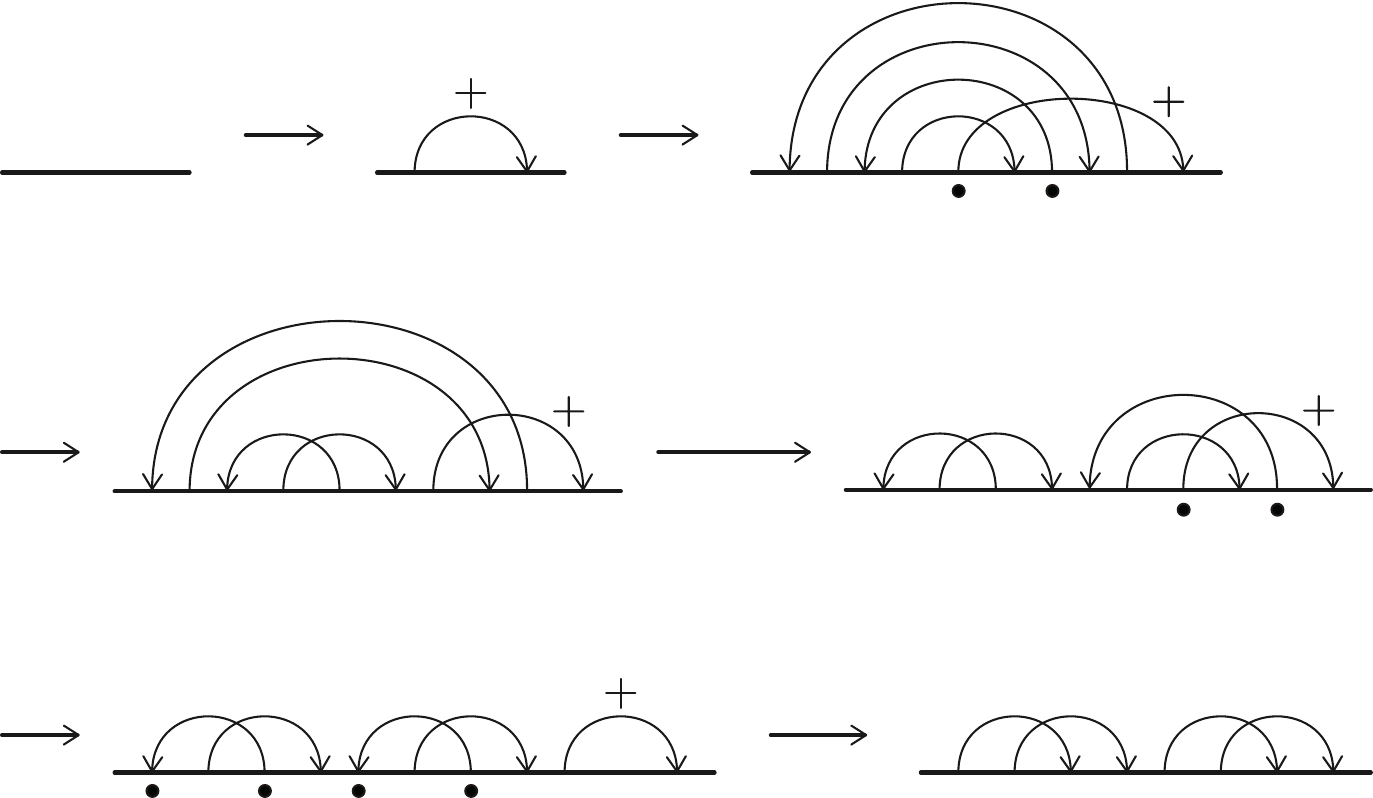}
    \put(7,124){$G=G_{0}$}
    \put(102,124){$G_{1}$}
    \put(223,124){$G_{2}$}
    \put(80,56){$G_{3}$}
    \put(247,56){$G_{4}$}
    \put(89,-12){$G_{5}$}
    \put(243,-12){$G_{6}=G'$}
    \put(63,156){I}
    \put(147,156){4}
    \put(5.5,84){$\Xi$}
    \put(149.5,84){Lem~\ref{lem-shellpair}}
    \put(5.5,19){$\Xi$}
    \put(177.5,19){$\Xi$, I}
    \put(216,183){$\e$}
    \put(216,174){$\e$}
    \put(216,166){$\e$}
    \put(216,157){$\e$}

    \put(56,85){$\e$}
    \put(82,85){$\e$}
    \put(75,111){$\e$}
    \put(75,102){$\e$}
    \put(205,85){$\e$}
    \put(231,85){$\e$}
    \put(267,85){$\e$}
    \put(267,94){$\e$}
    \put(38.5,21){$\e$}
    \put(64.5,21){$\e$}
    \put(86,21){$\e$}
    \put(112,21){$\e$}
    \put(222,21){$\e$}
    \put(248,21){$\e$}
    \put(269.5,21){$\e$}
    \put(295.5,21){$\e$}
   \end{overpic}
\vspace{1em}
\caption{Proof of Lemma~\ref{lem-censecutive} for $k=2$}
\label{pf-lem-censecutive}
\end{figure}
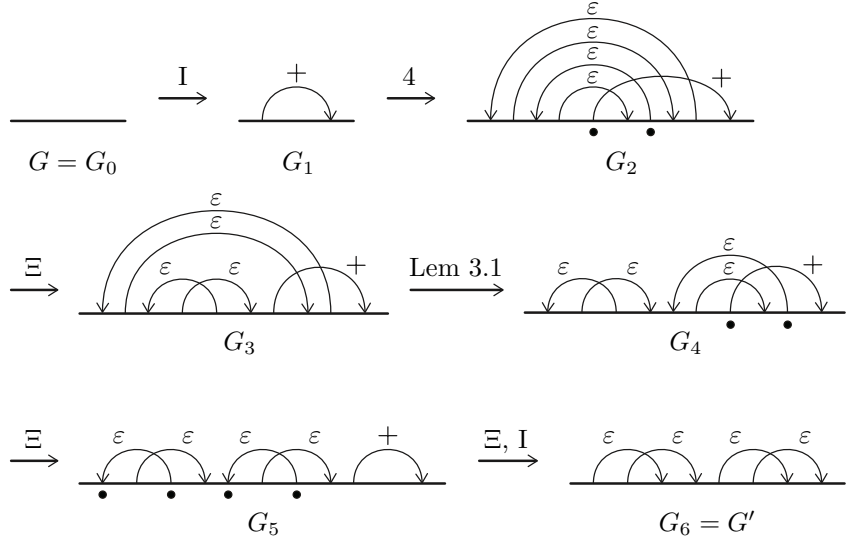

For an integer $a$, 
let $G(a)$ be the Gauss diagram in Figure~\ref{normalform}; 
that is, it consists of $|a|$ shell-pairs with sign $\e$, 
where $\e=1$ for $a>0$ and $\e=-1$ for $a<0$. 
In particular, $G(0)$ is the Gauss diagram with no chords. 
We denote by $K(a)$ the virtual knot represented by $G(a)$. 
We remark that $J(K(a))=2a$. 

\begin{figure}[htbp]
\centering
  \begin{overpic}[width=6cm]{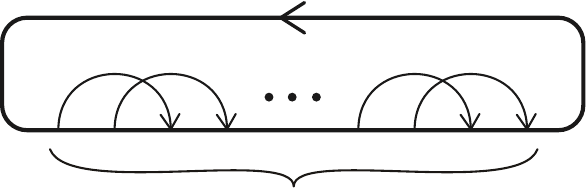}
    \put(20,34){$\e$}
    \put(59,34){$\e$}
    \put(107,34){$\e$}
    \put(146,34){$\e$}
    \put(55,-12){$|a|$ shell-pairs}
   \end{overpic}
\vspace{1em}
\caption{The Gauss diagram $G(a)$}
\label{normalform}
\end{figure}
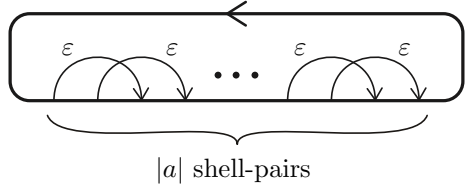

We give a normal form of an equivalence class of virtual knots under $2k$-moves and $\Xi$-moves as follows. 

\begin{proposition}\label{prop-normalform}
Any virtual knot $K$ is related to $K(a)$ for some $a\in\Z$ with $0\le a<k$ by a finite sequence of $2k$-moves and $\Xi$-moves. 
\end{proposition}

\begin{proof}
By \cite[Proposition 7.2]{ST}, any Gauss diagram $G$ of $K$ can be deformed into $G(a)$ for some $a\in\Z$ by a finite sequence of $\Xi$-moves and Reidemeister moves. 
If $a$ satisfies $0\le a<k$, then we have the conclusion. 

For $k\le a$, 
there is a unique positive integer $p$ such that $0\le a-pk<k$. 
Lemma~\ref{lem-censecutive} allows us to add $pk$ consecutive negative shell-pairs to $G(a)$. 
From the obtained Gauss diagram, we can remove $pk$ pairs of shell-pairs with opposite signs by Lemma~\ref{lem-shellpair}(ii) in order to obtain $G(a-pk)$.  
Therefore $G$ is related to $G(a-pk)$ by a finite sequence of $2k$-moves, $\Xi$-moves and Reidemeister moves. 

In the case $a<0$, let $q$ be the positive integer such that $0\le a+qk<k$. 
Using Lemmas~\ref{lem-shellpair}(ii) and \ref{lem-censecutive}, we add $qk$ consecutive positive shell-pairs to $G(a)$, 
and then remove $qk$ pairs of shell-pairs with opposite signs. 
Finally $G$ is related to $G(a+qk)$ by a finite sequence of $2k$-moves, $\Xi$-moves and Reidemeister moves. 
\end{proof}

Now we are ready to prove Theorem~\ref{thm-2kXi}. 

\begin{proof}[Proof of Theorem~\ref{thm-2kXi}]
(i)$\Rightarrow$(ii): 
By Proposition~\ref{prop-normalform}, $K$ and $K'$ are related to $K(a)$ and $K(a')$ for some $a,a'\in\Z$ with $0\le a,a'<k$, respectively, by a finite sequence of $2k$-moves and $\Xi$-moves. 
Then it follows from Theorems~\ref{thm-writhe}(ii) and~\ref{thm-ST} that 
\[
J(K)\equiv J(K(a))=2a \pmod{2k} 
\]
and
\[
J(K')\equiv J(K(a'))=2a' \pmod{2k}.
\] 
By assumption, we have $2a\equiv 2a' \pmod{2k}$.
Since the nonnegative integers $a$ and~$a'$ are less than $k$, we have $a=a'$. 
Therefore $K(a)$ and $K(a')$ coincide. 

(ii)$\Rightarrow$(i): 
This follows from Theorems~\ref{thm-writhe}(ii) and \ref{thm-ST}. 
\end{proof}

The following result is an immediate consequence of the proof of Theorem~\ref{thm-2kXi}. 

\begin{corollary}
A complete representative system of the equivalence classes of virtual knots under $2k$-moves and $\Xi$-moves is given by the set 
\[
\{K(a)\mid a\in\Z,\ 0\le a<k\}. 
\]
In particular, the number of equivalence classes is $k$. 
\qed
\end{corollary}

\section{$2k$-move distance}\label{sec-distance}
This section studies the $2k$-move distance for virtual knots. 
For two virtual knots $K$ and $K'$ that are related by a finite sequence of $2k$-moves, 
we denote by $\mathrm{d}_{2k}(K,K')$ the minimal number of $2k$-moves needed to deform a diagram of $K$ into that of $K'$.  
In particular, we set $\mathrm{u}_{2k}(K)=\mathrm{d}_{2k}(K,O)$, 
where $O$ is the trivial knot. 
We will show that for any virtual knot $K$ and any positive integer $a$, there is a virtual knot $K'$ such that $\mathrm{d}_{2k}(K,K')=a$ (Proposition~\ref{prop-distance}). 

In~\cite[Theorem~2.3]{JCK}, Jeong, Choi and Kim gave a lower bound for $\mathrm{d}_{2k}(K,K')$ using the affine index polynomials of $K$ and $K'$, which can be rephrased in terms of the $n$-writhes as follows. 

\begin{theorem}[{cf. \cite[Theorem~2.3]{JCK}}]\label{thm-lowerbound}
Let $K$ and $K'$ be virtual knots such that they are related by a finite sequence of $2k$-moves. 
Then we have 
\[
\mathrm{d}_{2k}(K,K')\geq\frac{1}{k}\sum_{n>0}|J_{n}(K)-J_{n}(K')|=\frac{1}{k}\sum_{n<0}|J_{n}(K)-J_{n}(K')|.
\] 
In particular, when $K'=O$ is the trivial knot, we have 
\[
\mathrm{u}_{2k}(K)\geq\frac{1}{k}\sum_{n>0}|J_{n}(K)|=\frac{1}{k}\sum_{n<0}|J_{n}(K)|.
\]
\end{theorem}

We remark that this theorem is a direct consequence of Lemma~\ref{lem-writhe}. 

In~\cite[Example~2.4]{JCK}, Jeong, Choi and Kim demonstrated that the lower bound for $\mathrm{u}_{2k}(K)$ in the theorem is sharp for some virtual knots $K$.  
However, they did not make it clear whether for a pair of nontrivial virtual knots $K$ and $K'$, the lower bound for $\mathrm{d}_{2k}(K,K')$  is sharp. 
We answer this by proving the following. 

\begin{proposition}\label{prop-distance}
Let $a$ be a positive integer. 
For any virtual knot $K$, 
there is a virtual knot $K'$ such that $\mathrm{d}_{2k}(K,K')=a$. 
\end{proposition}

\begin{proof}
Consider a long virtual knot diagram $T$ whose closure represents the virtual knot $K$. 
Let $K'$ be the virtual knot represented by the diagram $D$ in the left of Figure~\ref{pf-prop-distance}. 
The Gauss diagram $G_{D}$ associated with $D$ is given in the right of the figure, 
where the boxed part depicts the Gauss diagram $G_{T}$ corresponding to $T$. 

\begin{figure}[htbp]
\centering 
  \begin{overpic}[width=11cm]{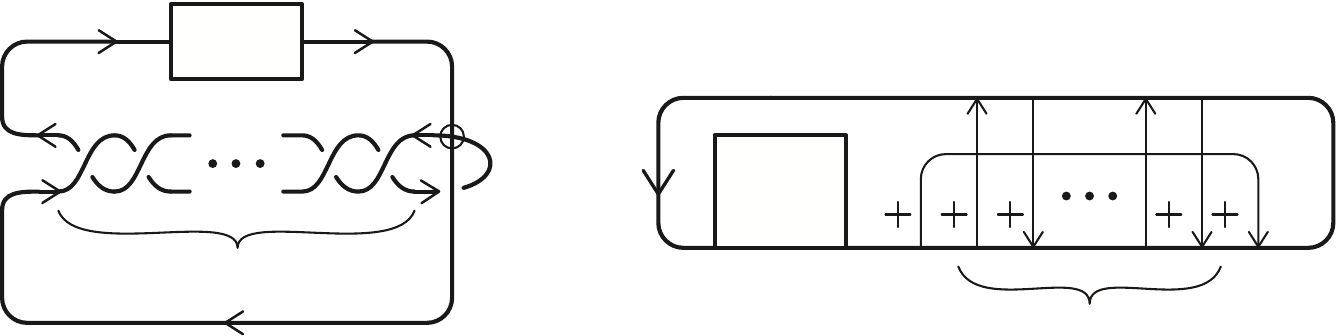}
    \put(52,65){$T$}
    \put(22,9){$2ak$ half-twists}
    \put(176,30){$G_{T}$}
    \put(234,-3){$2ak$ chords}
   \end{overpic}
\caption{A diagram $D$ of $K'$ and its Gauss diagram $G_{D}$}
\label{pf-prop-distance}
\end{figure}
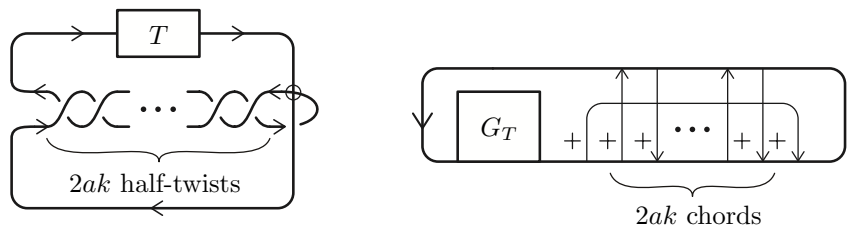

Removing $2ak$ half-twists from $D$ by $2k$-moves $a$ times, 
we can deform $D$ into a diagram of $K$. 
Therefore we have $\mathrm{d}_{2k}(K,K')\leq a$. 

The $2ak$ vertical chords in $G_{D}$ consist of $ak$ positive chords with index~$1$ and $ak$ positive chords with index $-1$, 
and the remaining one chord of $G_{D}$ excluding the chords in $G_{T}$ has index $0$. 
Hence it follows from \cite[Lemma~4.3]{ST} that 
\[
J_{n}(K')=
\begin{cases}
J_{1}(K)+ak & (n=1), \\
J_{-1}(K)+ak & (n=-1), \\
J_{n}(K) & (n\ne 0,\pm1). 
\end{cases}
\]
By Theorem~\ref{thm-lowerbound}, we have 
\[
\mathrm{d}_{2k}(K,K')\ge \frac{1}{k}|J_{1}(K)-J_{1}(K')|
=\frac{1}{k}|-ak|=a, 
\]
which shows $\mathrm{d}_{2k}(K,K')=a$. 
\end{proof}

We remark that for any positive integer $a$, 
there is a family of infinitely many virtual knots $\{K_{s}\}_{s=1}^{\infty}$ such that $\mathrm{u}_{2k}(K_{s})=a$ $(s\geq1)$. 
In fact, let $K_{s}$ be the virtual knot represented by the diagram in \cite[Figure~2.6]{NNSW} with $m=ak$. 
Then by Theorem~\ref{thm-lowerbound}, it can be seen that these virtual knots $K_{s}$'s form such a family.


\end{document}